\documentclass[a4paper,twoside,11pt,reqno]{amsart}
\usepackage{fullpage}
\usepackage[T1]{fontenc}
\usepackage[utf8]{inputenc}

\usepackage{hyperref}
\usepackage[slantedGreeks, partialup, noDcommand]{kpfonts}
\usepackage{graphicx}
\usepackage[mathscr]{euscript}
\usepackage{caption}
\usepackage{tikz}
\usepackage{appendix}
\usepackage{mathtools}
\mathtoolsset{showonlyrefs}

\usetikzlibrary{trees}

\hypersetup{ocgcolorlinks=true,allcolors=testc}
\hypersetup{
	colorlinks   = true,
	citecolor    = black
}
\hypersetup{linkcolor=black}
\hypersetup{urlcolor=black}

\newtheorem{theorem}{Theorem}
\newtheorem{lemma}[theorem]{Lemma}

\newtheorem{proposition}[theorem]{Proposition}

\newtheorem*{definition*}{Definition}

\newcommand{\E}{\mathbb{E}}

\newcommand{\Inf}{\textnormal{Inf}}
\newcommand{\mea}{\mu}
\newcommand{\spa}{E}
\newcommand{\real}{\mathbf{R}}

\newcommand{\Var}{\textnormal{Var}}
\newcommand{\CD}{\textnormal{CD}}

\newcommand{\I}{\mathcal{I}}
\newcommand{\id}{\textnormal{id}}

\def\bal{\begin{align*}}
\def\eal{\end{align*}}

\author{Paata Ivanisvili}
\address{(P.~I.) Department of Mathematics, University of California, Irvine, CA 92617, USA}
\email{pivanisv@uci.edu}

\author{Haonan Zhang}
\address{(H.~Z.)  Department of Mathematics, University of South Carolina, Columbia, SC, 29201, USA}
\email{haonanzhangmath@gmail.com}

\title{On the Eldan--Gross inequality}

\thanks{}

\begin{document}

	\maketitle
	
	\begin{abstract}
		A recent discovery of Eldan and Gross states that there exists a universal $C>0$ such that for all Boolean functions $f:\{-1,1\}^n\to \{-1,1\}$,
		$$
		\int_{\{-1,1\}^n}\sqrt{s_f(x)}d\mu(x) \ge C\textnormal{Var}(f)\sqrt{\log \left(1+\frac{1}{\sum_{j=1}^{n}\textnormal{Inf}_j(f)^2}\right)}
		$$
		where $s_f(x)$ is the sensitivity of $f$ at $x$, $\textnormal{Var}(f)$ is the variance of $f$, $\textnormal{Inf}_j(f)$ is the influence of $f$ along the $j$-th variable, and $\mu$ is the uniform probability measure. In this note, we give an alternative proof that applies to biased discrete hypercube, and spaces having positive Ricci curvature lower bounds in the sense of Bakry and \'Emery. 

	\end{abstract}
	
	\bigskip
	
	{\footnotesize
		\noindent {\em 2020 Mathematics Subject Classification.} 68R05; 47D07.
		
		\noindent {\em Key words.} Boolean functions. Isoperimetric inequality. Hypercontractivity. Bakry--\'Emery curvature-dimension condition.}

	 \section{Introduction}

Let $f:\{-1,1\}^n\to \mathbb{R}$, and  denote 
\begin{equation*}
	D_j f(x):=\frac{f(x)-f(x^{\oplus j})}{2},\qquad 1\le j\le n, 
\end{equation*}
where $x^{\oplus j}$ means flipping the $j$-th variable of $x$. Put
\begin{equation*}
	|\nabla f|:=\left(\sum_{j=1}^{n}|D_j f|^2\right)^{1/2}.
\end{equation*}
For Boolean functions $f :\{-1,1\}^{n} \to \{-1,1\}$ the square of {\em discrete gradient} $|\nabla f(x)|^2$ coincides with the \emph{sensitivity} of $f$ at $x$:
\begin{equation*}
|\nabla f(x)|^2=s_f(x):=\sharp \{1\le j\le n: f(x)\neq f(x^{\oplus j})\}.
\end{equation*}
Recall that the \emph{influence} $\Inf_j (f)$ of a Boolean function $f:\{-1,1\}^n\to \{-1,1\}$ along the $j$-th variable is the probability of ``the value of $f$ changes when flipping $x_j$", or equivalently,
\begin{equation*}
	\Inf_j (f)=\E|f-\E_{x_j}f|^p, 
\end{equation*}
where $\E$ is defined with respect to the uniform distribution $\mu_n$, and $p>0$ is arbitrary. 
As usual, $\Var(f)$ is understood as the variance of $f$ with respect to the uniform distribution:
\begin{equation*}
	\Var(f)=\E|f-\E f|^2=\E f^2-(\E f)^2.
\end{equation*}
Note that for Boolean functions  $f:\{-1,1\}^n\to \{-1,1\}$:
\begin{equation}\label{eq:boolean l1 l2}
	\|f-\E f\|_1=4a(1-a)=\Var(f)\qquad \textnormal{ with }\qquad a=\mu_n(\{f=1\}).
\end{equation}
Here and in what follows $\|f \|_p$ denotes the $L^p$-norm with respect to the uniform probability measure unless otherwise stated.

Recently, Eldan and Gross \cite{eg22concentration} proved the following inequality.
\begin{theorem}[\cite{eg22concentration}]\label{thm:eg}
	There exists a universal constant $C>0$ such that for all $n\ge 1$ and all Boolean function $f:\{-1,1\}^n\to \{-1,1\}$ we have 
	\begin{equation}\label{ineq:eg}
	\E|\nabla f|\ge C \Var(f)\sqrt{\log\left(1+\frac{e}{\sum_{j=1}^{n}\Inf_j(f)^2}\right)}.
	\end{equation}
\end{theorem}

This inequality was motivated by a conjecture of Talagrand  \cite{talagrand97boundaries} asked on the $p$-biased hypercube, i.e., 
\begin{equation}
\label{talerti}
\E\sqrt{h_{f}}\ge C_{p} \Var(f)\sqrt{\log\left(1+\frac{e}{\sum_{j=1}^{n}\Inf_j(f)^2}\right)},
\end{equation}
where $h_{f}$ is supported on a set $A =\{x : f(x)=1\}$, and at each point $x \in A$, the value $h_{f}$ counts the number of edges joining $x$ with $A^{c}$, i.e., complement of $A$. Clearly we always have $|\nabla f|  = \sqrt{h_{f}+h_{-f}}\geq \sqrt{h_{f}}$, therefore, the conjecture  (\ref{talerti}), if correct, would imply Eldan--Gross inequality (\ref{ineq:eg}).  However, in general $\mathbb{E}|\nabla f|$ and $\mathbb{E}\sqrt{h_{f}}$ may behave quite differently as $n\to \infty$, for example, if $A=\{(1, \ldots, 1)\}$ is a single point then  $\mathbb{E}|\nabla f| = \frac{n+\sqrt{n}}{2^{n}}\gg\frac{\sqrt{n}}{2^{n}}=\mathbb{E} \sqrt{h_{f}}$.

The proof of Eldan and Gross of the inequality (\ref{ineq:eg}) uses stochastic analysis and also leads to other impressive consequences. There are now several other proofs of \eqref{ineq:eg}; see for example van Handel  \cite{ramon}, Eldan--Kindler--Lifshitz--Minzer \cite{simpler23} and Beltran--Ivanisvili--Madrid  \cite[Remark 2]{BIM23sharp}. The main aim of this paper is to illustrate that some key ingredients of these proofs extend to more general spaces and the Boolean cube structure is not essential. Let us remark that the proof in \cite{ramon} can achieve the same goal after some modifications. Here, we highlight the use of an isoperimetric-type inequality by Bobkov and G\"{o}tze \cite{bg99} in 1999, the way of using which can be of independent interest.

We shall prove Theorem \ref{thm:eg} on the $p$-biased hypercubes; see Section \ref{sect:boolean proof} for details. The proof relies on Bobkov and G\"{o}tze's isoperimetric-type inequality \cite{bg99} and the hypercontractivity. 

\medskip

It might be easier to understand the proof in the continuous setting, and one typical example is the Gauss space. In fact, one may go further.
Let $(\spa,\mea,\Gamma)$ be a \emph{full Markov triple} with $(P_t)=(e^{-tL})$ being the associated diffuse Markov semigroup in the sense of \cite{bgl14book}. Here $E$ is a measure space with $\mu$ being some \emph{probability measure}, and $\Gamma$ is the \emph{carr\'e du champs} operator associated to $(P_t)=(e^{-tL})$:
$$\Gamma(f)=fL(f)-\frac{1}{2}L(f^2),$$
for suitable $f:E\to \real$.  By ``suitable" we refer to a certain class of functions for  which $\Gamma(f)$ is well-defined. In the sequel, we omit the rigorous definition to keep presentation compact, and we refer to \cite[Chapter 3]{bgl14book} for detailed definitions for $(\spa, \mea, \Gamma)$ being a full Markov triple. One crucial property here is the diffusion property, which may be understood as the chain rule for $L$, so that many functional inequalities follow from Bakry--\'Emery criterion
\begin{equation*}
	\Gamma (P_t f)\le e^{-2Kt}P_t\Gamma f
\end{equation*}
for all $f$ and $t\ge 0$. In this case, we say that the Ricci curvature is bounded from below by $K$, or it satisfies Bakry--\'Emery curvature-dimension condition $\CD(K,\infty)$. 
Examples of such include the heat semigroups on Riemannian manifolds  and $K$ can be chosen to be the lower bound of  Ricci curvature, which motivated the name in the above general framework. We will be interested in the case when the Ricci curvature lower bound $K$ is positive, for which a typical example is the  Gauss space with $(P_t)$ being the Ornstein--Uhlenbeck semigroup. In this case, the curvature lower bound is $K=1$. 

For any $n\ge 1$, denote $\spa^n$ the $n$-fold product of $(\spa,\mea)$ equipped with the product probability measure $\mea_n.$ We consider the tensor product of $(P_t)$ over $\spa^n$, which will be denoted by $(P_t)=(e^{-tL})$ again whenever no confusion can occur. For $f:\spa^n\to \real$, we define its variance 
\begin{equation*}
		\Var(f):=\E|f-\E f|^2=\E(f^2)-(\E f)^2,
\end{equation*}
with respect to $\mu_n$. Again, for Boolean functions $f:\spa^n\to\{-1,1\}$ one has
\begin{equation}\label{eq:boolean l1 l2 general}
	\|f-\E f\|_1=4a(1-a)=\Var(f)\qquad \textnormal{ with }\qquad a=\mu_n(\{f=1\}).
\end{equation}
Here and in what follows, for $f:E^n\to \real$, the $L^{p}$ norm $\|f\|_p$ is defined with respect to $\mu_n$. We define 
\begin{equation*}
	|\nabla f|:=\Gamma(f)^{1/2}
\end{equation*}
which is compatible with the discrete hypercube case by computing $\Gamma(f)$ for the heat semigroup.
For any Boolean function $f:\spa^n\to\{-1,1\}$ and $1\le j\le n$, we define the \emph{influence of $j$-th variable} as
\begin{equation*}
	\Inf_j(f):=\E|f-\E_{x_j}(f)|.
\end{equation*}
There are different definitions of influences in this context; our definition (also known as \emph{$L^1$-influence} or \emph{geometric influence}) agrees with the one in discrete hypercube case explained above.

\begin{theorem}\label{thm:general}
	Let $(\spa,\mea,\Gamma)$ be a full Markov triple satisfying Bakry--\'Emery curvature-dimension condition $\CD(K,\infty)$ with $\mu$ a probability measure and $0<K<\infty$. Then there exist constants $C=C_K, C'=C_K'$ depending only on $K$ such that for all $n\ge 1$ and all Boolean functions $f:\spa^n\to\{-1,1\}$
		\begin{equation}\label{ineq:talagrand ctns}
	\E|\nabla f|\ge C_K \Var(f)\sqrt{\log\left(\frac{e}{\Var(f)}\right)}
	\end{equation}
	and 
	\begin{equation}\label{ineq:eg ctns}
	\E|\nabla f|\ge C_K' \Var(f)\sqrt{\log\left(1+\frac{e}{\sum_{j=1}^{n}\Inf_j(f)^2}\right)}.
	\end{equation}
\end{theorem}

The inequality \eqref{ineq:talagrand ctns} was inspired by Talagrand's inequality \cite{talagrand93}. Although this inequality can be directly derived from the stronger, so-called {\em local Bobkov inequality} well-known in the literature, we will provide an "alternative" proof in Section \ref{sect:general case}. This proof, while slightly longer, incorporates the {\em local Bobkov inequality} as one of its key components. This approach might seem excessive, but it is intentional, as the proof of \eqref{ineq:eg ctns} will closely follow the structure of the "alternative" proof of \eqref{ineq:talagrand ctns}, and additioanly,  the inequality (\ref{ineq:talagrand ctns}) itself  will be used in our proof of \eqref{ineq:eg ctns}.

\subsection*{Acknowledgment}
P.I. is supported by National Science Foundation CAREER-DMS-2152401. H.Z. is supported by National Science Foundation DMS-2453408.

	 \section{Boolean cube case}\label{sect:boolean proof}
	 
	In this section, we give an alternative proof of Theorem \ref{thm:eg}. 
The proof also extends to the biased case which we now explain. 
For any $p\in (0,1)$ we equip $\{-1,1\}$ with the probability measure $\mu_p=p\delta_{1}+(1-p)\delta_{-1}$. We use $\|f\|_q$ to denote the $L^q$-norm of $f:\{-1,1\}^n\to \real$ with respect to $\mu_p^{\otimes n}$. We write $\E=\E_{\mu_p^{\otimes n}}$ for short, and consider 
\begin{equation}\label{eq: defn of semigroup biased}
P_t=(e^{-t}\id+(1-e^{-t})\E)^{\otimes n},\qquad t\ge 0
\end{equation}
which is a semigroup of unital positive linear operators over $\real^{\{-1,1\}^n}$ such that $P_0(f)=f$ and $\lim_{t\to\infty} P_t (f)=\E f$. Accordingly, we write $\Var(f)$ for $\Var_{\mu_p}(f)$ with respect to this biased $\mu_p$. 

For any $f:\{-1,1\}^n\to \{-1,1\}$ and $1\le j\le n$, we define the $j$-influence of $f$ as 
\begin{equation}\label{eq:defn of biased influence}
\Inf_j(f):=\|f-\E_{x_j\sim \mu_p} f\|_1.
\end{equation}
In a literature of Boolean functions, in general the quantity  $\|D_jf\|_1$ is defined as the  ($L^1$) influence on the biased hypercube. Notice that, for $f:\{-1,1\}^n\to \{-1,1\}$ one has 
$$
\|f-\E_{x_j\sim \mu_p} f\|_1=4p(1-p)\|D_jf\|_1,
$$
and the constant factor of $4p(1-p)$ will make no essential difference for our discussion. So we keep our choice \eqref{eq:defn of biased influence}.
We still put 
$$
|\nabla f(x)|:=\left(
\sum_{j=1}^{n}|D_jf(x)|^2\right)^{1/2}.
$$
Then the main result of this section reads as follows.

\begin{theorem}\label{thm:eg biased}
	Fix $p\in (0,1)$ and let the notation be as above. Then there exists a constant $C_p>0$ such that for all $n\ge 1$ and all Boolean functions $f:\{-1,1\}^n\to \{-1,1\}$ we have 
	\begin{equation}\label{ineq:eg biased}
	\E|\nabla f|\ge C_p \Var(f)\sqrt{\log\left(1+\frac{e}{\sum_{j=1}^{n}\Inf_j(f)^2}\right)}.
	\end{equation}
\end{theorem}

One ingredient of the proof can be derived from an isoperimetric-type inequality of Bobkov and G\"otze \cite{bg99}. The proof is inspired by the unbiased case \cite{learning}. More discussions can be found in Appendix \ref{appendix on Bobkov}. Similar estimates were obtained in \cite[Claim 8, items 1-2]{ramon} via a more direct argument.

\begin{lemma}\label{lem: biased isoperimetry}
Let $p\in (0,1).$ Suppose that $t\ge 0$ and $f:\{-1,1\}^n\to \{-1,1\}$. Then 
\begin{equation}
	\Var(f)-\Var(P_t f)=1-\E|P_tf|^2\le 2\max\{p,1-p\}\sqrt{\pi(1-e^{-2t})}\E|\nabla f|.
\end{equation}
In particular, sending $t\to\infty$, we have 
\begin{equation}\label{ineq:isoperimetry boolean cube biased}
\|f-\E f\|_1=\Var(f)\le 2\max\{p,1-p\}\sqrt{\pi}\E|\nabla f|.
\end{equation}
\end{lemma}

\begin{proof}
According to \cite[Theorem 2.3]{bg99}, for probability spaces $(\Omega_i,\nu_i), 1\le i\le n$ and their product $(\Omega,\nu)$, we have 
\begin{equation}\label{ineq:bg general space}
\I(\E_{\nu}h)\le \sqrt{2}\E_{\nu}\sqrt{\sum_{i}\Var_{\nu_i}(h)}
\end{equation}
for any measurable  $h:\Omega\to \{0,1\}$, where $\I$ is the Gaussian isoperimetric profile $\I:[0,1]\to [0,1/\sqrt{2\pi}]$ given by $\I=\varphi\circ \Phi^{-1}$ with $\Phi(r)=\int_{-\infty}^{r} \varphi(x)dx$ and $\varphi(x)=\frac{1}{\sqrt{2\pi}}e^{-x^2/2}$.

Now we choose $\Omega_i=\{-1,1\}$. The probability measure will be chosen as follows. Fix $x\in \{-1,1\}^n$ and $t\ge 0$, and consider the following linear functional $\varphi_i$ 
$$
\varphi_i(g):=P_t(g)(x_i)
=(1-e^{-t})[pg(1)+(1-p)g(-1)]+e^{-t}g(x_i),\qquad g:\Omega_i=\{-1,1\}\to \real,
$$
where $P_t$ is exactly the depolarizing semigroup \eqref{eq: defn of semigroup biased} in dimension $n=1$.
Then $\varphi_i$ is positive and $\|\varphi_i\|=\varphi_i(1)=1$. By Riesz representation theorem, there exists a probability measure $\nu_i$ over $\{-1,1\}$ such that 
$$\E_{\nu_i}(g)=\int_{\{-1,1\}}gd\nu_i
=P_t(g)(x_i),
\qquad g:\{-1,1\}\to \real.$$
Therefore, in general dimension $n\ge 1$ one has 
\begin{equation}\label{eq:expectation of nu}
\E_\nu(g)=P_t(g)(x),\qquad g:\{-1,1\}^n\to\real
\end{equation}
for $P_t$ in \eqref{eq: defn of semigroup biased}.
Moreover, for $h:\{-1,1\}^n\to \{0,1\}$ we have
\begin{align*}
\E_{\nu_i}(h(y))
&=(1-e^{-t})[ph(y_1,\dots, 1,\dots, y_n)+(1-p)h(y_1,\dots, -1,\dots, y_n)]+e^{-t}h(y_1,\dots, x_i,\dots, y_n)\\
&=
(1-e^{-t})(pa+qb)+e^{-t}c
\end{align*}
where $a=h(y_1,\dots, 1,\dots, y_n),b=h(y_1,\dots, -1,\dots, y_n)$ and  $c=h(y_1,\dots, x_i,\dots, y_n)$. Here we also write $q=1-p$ and let us denote $r:=\max\{p,q\}$. Thus using $a=a^2,b=b^2,c=c^2$ we obtain by simple computations: if $c=a$, then
\begin{equation}
\Var_{\nu_i}(h(y))=\E_{\nu_i}(h(y))-\E_{\nu_i}(h(y))^2
=q(1-e^{-t})(p+qe^{-t})(a-b)^2\le r^2(1-e^{-2t})(a-b)^2,
\end{equation}
while if  $c=b$, then
\begin{equation}
\Var_{\nu_i}(h(y))=\E_{\nu_i}(h(y))-\E_{\nu_i}(h(y))^2
=p(1-e^{-t})(q+pe^{-t})(a-b)^2\le r^2(1-e^{-2t})(a-b)^2.
\end{equation}
So in either case, we have
\begin{equation}
\Var_{\nu_i}(h(x))
\le r^2 (1-e^{-2t})|h(x_1,\dots, x_i,\dots, x_n)-h(x_1,\dots, -x_i,\dots, x_n)|^2
=4r^2(1-e^{-2t})|D_i h(x)|^2
\end{equation}
and thus for all $h:\{-1,1\}^n\to \{0,1\}$
\begin{equation}\label{ineq:sum of variance nu}
\sum_{i=1}^{n}\Var_{\nu_i}(h(x))
\le 4r^2(1-e^{-2t})\sum_{i=1}^{n}|D_i f(x)|^2
=4r^2(1-e^{-2t})|\nabla h(x)|^2.
\end{equation}

Combining \eqref{ineq:bg general space}, \eqref{eq:expectation of nu} and \eqref{ineq:sum of variance nu}, one has the pointwise inequality
\begin{equation}\label{ineq: biased isoperimetric profile}
\I(P_t h)\le 2\sqrt{2}\max\{p,1-p\}\sqrt{1-e^{-2t}}P_t|\nabla h|,\qquad h:\{-1,1\}^n\to \{0,1\}.
\end{equation}
For any $f:\{-1,1\}^n\to \{-1,1\}$ we set $h:=\frac{1}{2}(1+f)$ that takes values in $\{0,1\}.$ Then we just proved 
\begin{equation}
\I\left(\frac{1+P_t f}{2}\right)\le \sqrt{2}\max\{p,1-p\}\sqrt{1-e^{-2t}}P_t|\nabla f|,\qquad f:\{-1,1\}^n\to \{-1,1\}.
\end{equation}
This, together with the the elementary inequality 
\begin{equation}\label{ineq:lower bound of isoperimetric profile}
	\I (s)\ge \sqrt{\frac{2}{\pi}}\left[\frac{1}{2}-2\left(\frac{1}{2}-s\right)^2\right],\qquad s\in [0,1]
\end{equation}
yields 
\begin{equation*}
1-(P_t f)^2\le 2\max\{p,1-p\}\sqrt{\pi(1-e^{-2t})}P_t|\nabla f|,\qquad f:\{-1,1\}^n\to \{-1,1\}.
\end{equation*}
Taking the expectation gives the desired estimate. 
\end{proof}

Another ingredient is the hypercontractivity on the biased hypercube. 

\begin{proposition}\label{prop:hypercontractivity biased}
For $p\in (0,1)$ we consider the semigroup $P_t$ defined in \eqref{eq: defn of semigroup biased}. Then for any $1<r<2$ and for any $f:\{-1,1\}^n\to \real$ we have
\begin{equation}\label{ineq:hypercontractivity biased}
\|P_t f\|_2\le \|f\|_r\qquad \textnormal{ for }\qquad t\ge 2p(1-p)\log\frac{1}{r-1}.
\end{equation}
\end{proposition}

When $p=1/2$, this is a well-known result by Bonami \cite{bonami70} and the time in  \eqref{ineq:hypercontractivity biased} is optimal. In the general biased case, the optimal time is more involved \cite{ole03}. We give a proof of \eqref{ineq:hypercontractivity biased} in the Appendix \ref{appendix on hypercontractivity} for convenience. 


From the above proposition, one may deduce the following lemma whose proof is essentially the same as that of \cite[Lemma 11]{ramon}. 

\begin{lemma}\label{lem: martingale boolean W}
Fix $p\in (0,1)$. For all $t\ge 0$ and for all $f:\{-1,1\}^n\to \real$, we have 
\begin{equation}
\Var(P_t f)\le \left(\sum_{j=1}^{n}\Inf_j(f)^2\right)^{\theta(t)}	\Var(f)^{1-\theta(t)},
\end{equation}
where $\theta(t)=\frac{1-e^{-2Kt}}{1+e^{-2Kt}}$ with $K=\frac{1}{4p(1-p)}$.
\end{lemma}

\begin{proof}
For any $f$, consider ($\E_{x_j}:=\E_{x_j\sim \mu_p}$) 
\begin{equation*}
	f^{(0)}=f, \qquad f^{(j)}=\E_{x_1,\dots, x_j}f, \quad 1\le j\le n.
\end{equation*}
Write $f_t:=P_t(f)$. Then we have the martingale property
\begin{equation}\label{eq:martingale boolean}
	\Var(f_t)=\|f_t^{(0)}-f_t^{(n)}\|^2_2
	=\sum_{j=1}^{n}\|f_t^{(j-1)}-f_t^{(j)}\|_2^2.
\end{equation}
Note that $\E_{x_j}P_t=P_t\E_{x_j}$, so by \eqref{eq:martingale boolean}, hypercontractivity \eqref{ineq:hypercontractivity biased} and H\"older's inequality
\begin{align*}
		\Var(f_t)
	=\sum_{j=1}^{n}\|f_t^{(j-1)}-f_t^{(j)}\|_2^2
	&=\sum_{j=1}^{n}\|P_t(f^{(j-1)}-f^{(j)})\|_2^2\\
	&\le \sum_{j=1}^{n}\|f^{(j-1)}-f^{(j)}\|_{1+e^{-2Kt}}^2\\
	&\le \sum_{j=1}^{n}\|f^{(j-1)}-f^{(j)}\|_{1}^{2\theta(t)}\|f^{(j-1)}-f^{(j)}\|_{2}^{2-2\theta(t)}\\
	&\le\left(\sum_{j=1}^{n}\|f^{(j-1)}-f^{(j)}\|_{1}^2\right)^{\theta(t)}\left(\sum_{j=1}^{n}\|f^{(j-1)}-f^{(j)}\|_{2}^{2}\right)^{1-\theta(t)}\\
	&=\left(\sum_{j=1}^{n}\|f^{(j-1)}-f^{(j)}\|_{1}^2\right)^{\theta(t)}	\Var(f)^{1-\theta(t)},
\end{align*}
where in the last equality we used \eqref{eq:martingale boolean} again. 
Since 
\begin{equation*}
\|f^{(j-1)}-f^{(j)}\|_{1}=\|\E_{x_1,\dots, x_{j-1}}(f-\E_{x_j}f)\|_1\le \E | f-\E_{x_j}f|=\Inf_j(f),
\end{equation*}
we get
\begin{equation*}
		\Var(f_t)\le \left(\sum_{j=1}^{n}\Inf_j(f)^2\right)^{\theta(t)}	\Var(f)^{1-\theta(t)}.\qedhere
\end{equation*}
\end{proof}

Now we recall the following inequality due to \cite[Theorem 1.1]{talagrand93}, which will directly follow from a stronger {\em local Bobkov inequality}. In the next section one may find an alternative proof of it, though presented in the continuous setting. 

\begin{proposition}\label{thm: talagrand ineq biased}
	Fix $p\in (0,1)$. Then there exists a universal constant $C_p>0$ such that for all $n\ge 1$ and all Boolean function $f:\{-1,1\}^n\to \{-1,1\}$ we have 
	\begin{equation}\label{ineq:log var}
		\E|\nabla f|\ge C_p \Var(f)\sqrt{\log\left(\frac{e}{\Var(f)}\right)}.
	\end{equation}
\end{proposition}
\begin{proof}
It is equivalent to proving the inequality for $h:\{-1,1\}^n\to \{0,1\}$. Starting from \eqref{ineq: biased isoperimetric profile}, we obtain the isoperimetric inequality by sending $t\to \infty$:
\begin{equation}\label{ineq:isoperimetric ctns}
\I(\E h)\le 2\sqrt{2}\max\{p, 1-p\}\E |\nabla h|,\qquad h: \{-1,1\}^n\to \{0,1\}
\end{equation} 
Now instead of \eqref{ineq:lower bound of isoperimetric profile}, we employ the following better lower bound 
\begin{equation}\label{ineq:better lower bound of I}
\I(x)\ge C x(1-x)\sqrt{\log\frac{1}{x(1-x)}},\qquad x\in (0,1).
\end{equation}
that entails 
\begin{equation}\label{ineq:lower bound I of expectation}
\I(\E h)\ge C \Var(h)\sqrt{\log\frac{1}{\Var(h)}},\qquad h: \{-1,1\}^n\to \{0,1\}
\end{equation}
since $\Var(h)=\E h(1-\E h)$. Combining \eqref{ineq:isoperimetric ctns} and \eqref{ineq:lower bound I of expectation} finishes the proof. To see \eqref{ineq:better lower bound of I}, it suffices to consider the cases when $x\downarrow 0$ and $x\uparrow 1$. Note that $\I(x)=\I(1-x)$, so it suffices to show \eqref{ineq:better lower bound of I} for $x\downarrow 0$ and this is known \cite[Remark 1]{bobkov97}.
\end{proof}

Now we are ready to prove Theorem \ref{thm:eg biased}.

\begin{proof}[Proof of Theorem \ref{thm:eg biased}]
Denote $W(f):=\sum_{j=1}^{n}\Inf_j(f)^2$. Recalling \eqref{ineq:isoperimetry boolean cube biased}, the desired inequality \eqref{ineq:eg biased} follows immediately when $W(f)$ is large enough, say $W(f)\ge 1/100$. So let us assume $W(f)< 1/100$.

In view of \eqref{ineq:log var}, we may assume $\Var(f)\ge W(f)^\alpha$ with $\alpha>0$. In fact, if $\Var(f)< W(f)^\alpha$, then our desired \eqref{ineq:eg} will follow from \eqref{ineq:log var} because (recalling $W(f)< 1/100$)
\begin{equation*}
	\log\left(\frac{e}{\Var(f)}\right)
	=1+\log\left(\frac{1}{\Var(f)}\right)
	> 1+\alpha\log\left(\frac{1}{W(f)}\right)
	\ge C\min\{1,\alpha\}\log\left(1+\frac{e}{W(f)}\right).
\end{equation*}
We may also assume $\Var(f)\ge 100W(f)$. Otherwise, combining $\Var(f)< 100W(f)$ and \eqref{ineq:log var}, to prove the desired inequality \eqref{ineq:eg} it suffices to show 
\begin{equation}
\log \left(\frac{e}{100W(f)}\right)
=\log\left(\frac{1}{100}\right)+\log\left(\frac{e}{W(f)}\right)
\ge C\log \left(1+\frac{e}{W(f)}\right)
\end{equation}
which is true with universal $C>0$ since $W(f)< 1/100$.

So from now on we assume $\Var(f)\ge \sqrt{W(f)}$ and $\Var(f)\ge 100 W(f)$. According to Lemmas \ref{lem: biased isoperimetry} and \ref{lem: martingale boolean W},
\begin{equation}
\E|\nabla f|
\ge \frac{\Var(f)-\Var(P_t f)}{2\max\{p,1-p\}\sqrt{\pi(1-e^{-2t})}}
\ge \frac{\Var(f)-W(f)^{\theta(t)}\Var(f)^{1-\theta(t)}}{2\max\{p,1-p\}\sqrt{\pi(1-e^{-2t})}}.
\end{equation}
where $\theta(t)=\frac{1-e^{-2Kt}}{1+e^{-2Kt}}$ with $K=\frac{1}{4p(1-p)}\ge 1$.
Thus, to prove \eqref{ineq:eg biased} it remains to show
\begin{equation}\label{ineq:sufficient boolean t}
	\frac{1-\left(\frac{W(f)}{\Var(f)}\right)^{\theta(t)}}{\sqrt{1-e^{-2Kt}}}\ge C\sqrt{\log\left(1+\frac{e}{W(f)}\right)}
\end{equation}
for some $t>0$.
Writing $\epsilon=1-e^{-2Kt}$, \eqref{ineq:sufficient boolean t} becomes
 \begin{equation}\label{ineq:sufficient boolean eps}
 	\frac{1-\left(\frac{\Var(f)}{W(f)}\right)^{\frac{\epsilon}{\epsilon-2}}}{\sqrt{\epsilon}}\ge C\sqrt{\log\left(1+\frac{e}{W(f)}\right)}
 \end{equation}
for some $\epsilon\in (0,1)$. Recalling $\Var(f)\ge 100W(f)$, we may choose $\epsilon^{-1}=\log\left(\Var(f)/W(f)\right)\in (1,\infty)$, so that
 \begin{equation*}
 	\frac{1-\left(\frac{\Var(f)}{W(f)}\right)^{\frac{\epsilon}{\epsilon-2}}}{\sqrt{\epsilon}}
 	=\frac{1-e^{\frac{1}{\epsilon-2}}}{\sqrt{\epsilon}}
 	\ge C\sqrt{1+\frac{1}{\epsilon}}
 	=C\sqrt{1+\log\left(\Var(f)/W(f)\right)},
 \end{equation*}
 where we used the fact that 
\begin{equation}\label{ineq:numerical fact}
\min_{\epsilon\in [0,1]}\frac{1-e^{\frac{1}{\epsilon-2}}}{\sqrt{1+\epsilon}} \ge \frac{1-e^{-1/2}}{\sqrt{2}}>0.
\end{equation} 
Recall that $\Var(f)\ge \sqrt{W(f)}$. Therefore,
 \begin{equation}
 \sqrt{1+\log\left(\Var(f)/W(f)\right)}
 \ge \sqrt{1+\frac{1}{2}\log\left(1/W(f)\right)}
 \ge C\sqrt{\log\left(1+\frac{e}{W(f)}\right)}.
 \end{equation}
This concludes the proof of \eqref{ineq:sufficient boolean eps}.
\end{proof}

\section{Continuous analogues}\label{sect:general case}

In this section we prove Theorem \ref{thm:general}. The main condition is that the Markov triple $(E,\mu,\Gamma)$ satisfies the \emph{Bakry--\'Emery curvature-dimension condition $\CD(K,\infty)$} with $\mu$ a probability measure and $K>0$. We refer to \cite{bgl14book} for more details. 

We collect here  several useful corollaries of $\CD(K,\infty)$ with $K>0$. 

\begin{proposition}\label{prop:consequences of Ric}
	Fix a Markov triple $(\spa,\mea,\Gamma)$ with the diffuse Markov semigroup $P_t=e^{-tL}$. Assume that it satisfies curvature-dimension condition $\CD(K,\infty)$ with $K>0$. 
	\begin{enumerate}
		\item Hypercontractivity \cite[Theorem 5.2.3, Proposition 5.7.1]{bgl14book}: for $f:\spa\to \real$, $t\ge 0$ and $1<p\le q<\infty$  we have 
		\begin{equation}\label{ineq:hyper}
			\|P_tf\|_q\le \|f\|_p \qquad \textnormal{ for all }\qquad t\ge \frac{1}{2K}\log\frac{q-1}{p-1}.
		\end{equation}
	\item (Local) Bobkov inequality \cite[Theorem 8.5.3]{bgl14book}: for $g:\spa\to \{0,1\}$ and $t\ge 0$ we have under suitable approximation
	\begin{equation}\label{ineq:local bobkov}
		\I (P_t g)\le \sqrt{\frac{1-e^{-2Kt}}{K}}P_t\sqrt{\Gamma(g)},
	\end{equation}
	where $\I$ is the Gaussian isoperimetric profile. 
	\end{enumerate}
\end{proposition}
 Here, \eqref{ineq:local bobkov} is a special case of \cite[Theorem 8.5.3]{bgl14book} that suffices for our use. In fact, the original \cite[Theorem 8.5.3]{bgl14book} states that for $\alpha\ge 0,t\ge 0$ and measuable functions $f:E\to \real$
 \begin{equation}\label{original bobkov}
     \sqrt{\I^2(P_t f)+\alpha\Gamma(P_t f)}
     \le P_t\left(\sqrt{\I^2( f)+c_\alpha(t)\Gamma(f)}\right),\qquad c_\alpha(t)=\frac{1-e^{-2Kt}}{K}+\alpha e^{-2Kt}.
 \end{equation}
 Then \eqref{ineq:local bobkov} follows by taking $\alpha=0$ in \eqref{original bobkov} and approximating the characteristic function $g:E\to\{0,1\}$ in \eqref{ineq:local bobkov} with measurable functions $f$ in \eqref{original bobkov}. Here, the Boolean property plays a role, and we used $\I(0)=\I(1)=0$.
 For $n\ge 1$, we shall still denote by $P_t$ its $n$-fold products on $E^n$. Then the above results actually hold for $E^n$ after a tensorization argument, which we will use directly.

The key to the proof of Lemma \ref{lem: biased isoperimetry} is \eqref{ineq: biased isoperimetric profile} that can be considered as a discrete analog of \eqref{ineq:local bobkov}. So following the same arguments in the proof Lemma \ref{lem: biased isoperimetry}, we may deduce 
\begin{equation*}
	1-(P_t f)^2\le\sqrt{\frac{\pi(1-e^{-2Kt})}{2K}}P_t|\nabla f|,\qquad f:E^n\to \{-1,1\}
\end{equation*}
from \eqref{ineq:local bobkov}.
Taking the expectation, we get
\begin{equation}\label{ineq:bg general}
	\Var(f)-\Var(P_t f)=1-\E|P_t f|^2\le\sqrt{\frac{\pi(1-e^{-2Kt})}{2K}}\E|\nabla f|.
\end{equation}
Sending $t\to \infty$ yields
\begin{equation}\label{ineq: l1 poincare general}
	\|f-\E f\|_1=\Var(f)=1-\E|P_t f|^2\le\sqrt{\frac{\pi}{2K}}\E|\nabla f|.
\end{equation}


\begin{lemma}\label{lem: martingale general W}
Under the above assumptions, for $f:E^n\to \real$, we have 
\begin{equation}
\Var(P_t f)\le \left(\sum_{j=1}^{n}\Inf_j(f)^2\right)^{\theta_K(t)}	\Var(f)^{1-\theta_K(t)},
\end{equation}
for all $t\ge 0$ with $\theta_K(t)=\frac{1-e^{-2Kt}}{1+e^{-2Kt}}$.
\end{lemma}

Using the hypercontractivity estimate \eqref{ineq:hyper}, the proof of the above lemma is the same as that of Lemma \ref{lem: martingale boolean W}.

%

\begin{proof}[Proof of Theorem \ref{thm:general}]
We first show \eqref{ineq:talagrand ctns}, which may follow from the same proof of Proposition \ref{thm: talagrand ineq biased}. Here we present another proof and the argument is similar to the proof of Theorem \ref{thm:eg biased}.  When $\Var(f)\ge  1/100$, then the desired \eqref{ineq:talagrand ctns} follows immediately from \eqref{ineq: l1 poincare general}. 

Now we assume $\Var(f)< 1/100$. By hypercontractivity \eqref{ineq:hyper} and H\"older's inequality
\begin{equation}
\Var(P_t f)=\|P_t(f-\E f)\|_2^2
\le \|f-\E f\|_{1+e^{-2Kt}}^2
\le  \|f-\E f\|_{1}^{2\theta_K(t)} \|f-\E f\|_{2}^{2-2\theta_K(t)}
=\Var(f)^{1+\theta_K(t)}.
\end{equation}
In the last equality we used the fact 
$$\|f-\E f\|_1=\|f-\E f\|_2^2=\Var(f) \qquad \textnormal{ for }\qquad f:E^n\to \{-1,1\}.$$
This, together with \eqref{ineq:bg general}, yields
\begin{equation}
\E |\nabla f|\ge \sqrt{\frac{2K}{\pi}}\frac{\Var(f)-\Var(P_t f)}{\sqrt{1-e^{-2Kt}}}
\ge \sqrt{\frac{2K}{\pi}}\frac{\Var(f)-\Var(f)^{1+\theta_K(t)}}{\sqrt{1-e^{-2Kt}}},\qquad t\ge 0.
\end{equation}
Thus, to prove \eqref{ineq:talagrand ctns}, it suffices to show
\begin{equation}\label{ineq: general t}
\frac{1-\Var(f)^{\theta_K(t)}}{\sqrt{1-e^{-2Kt}}}
\ge C\sqrt{\log\left(\frac{e}{\Var(f)}\right)}
\end{equation}
for some $t>0$.
Write $\epsilon=1-e^{-2Kt}$, and \eqref{ineq: general t} becomes
\begin{equation}\label{ineq: general eps}
\frac{1-e^{\frac{\epsilon}{\epsilon-2}\log\left(\frac{1}{\Var(f)}\right)}}{\sqrt{\epsilon}}
\ge C\sqrt{\log\left(\frac{e}{\Var(f)}\right)}
\end{equation}
for some $\epsilon\in (0,1)$.
Recall that $\Var(f)< 1/100$, so we may choose $\epsilon^{-1}=\log\left(\frac{1}{\Var(f)}\right)\in (1,\infty)$ and the proof is reduced to $\min_{\epsilon\in [0,1]}\frac{1-e^{\frac{1}{\epsilon-2}}}{\sqrt{1+\epsilon}} >0$. This is true as we argued in \eqref{ineq:numerical fact}.

Therefore, we complete the proof of \eqref{ineq:talagrand ctns}. Combining this with \eqref{ineq: l1 poincare general} and Lemma \ref{lem: martingale general W}, one may establish \eqref{ineq:eg ctns} following exactly the same proof of Theorem \ref{thm:eg} in the last section. 
\end{proof}


\appendix

\section{Proof of Proposition \ref{prop:hypercontractivity biased}}
\label{appendix on hypercontractivity}
Recall that $p\in (0,1)$ and $\mu_p(\{1\})=p,\mu_p(\{-1\})=1-p$. The generator of $P_t$ in dimension one is given by $L(f)=\mu_p(f)-f$ for $f:\{-1,1\}\to \real$. Put $q:=1-p$. We shall prove the following logarithmic Sobolev inequality: for all non-negative $f:\{-1,1\}\to\real$ we have 
\begin{equation}\label{ineq: lsi biased}
\mu_p(f^2\log f^2)-\mu_p(f^2)\log \mu_p(f^2)\le -\frac{1}{2pq}\mu_p(fL(f)).
\end{equation}
Then by standard arguments \cite{gross}, it implies 
\begin{equation}
\|P_t(f)\|_{r_2}\le \|f\|_{r_1},\qquad t\ge 2pq\log\frac{r_2-1}{r_1-1}
\end{equation}
for all $f:\{-1,1\}^n\to\real$ and for all $1<r_1<r_2<\infty$ which finishes the proof of Proposition \ref{prop:hypercontractivity biased}.

Now it remains to prove \eqref{ineq: lsi biased}. By homogeneity, we may assume that $f(x)=1+sx$ with $s\in [-1,1]$. Then \eqref{ineq: lsi biased} becomes
$$
p(1+s)^2\log (1+s)^2+q(1-s)^2\log (1-s)^2-[p(1+s)^2+q(1-s)^2]\log [p(1+s)^2+q(1-s)^2]\le 2s^2,
$$
for $s\in [-1,1]$ by direct computations. Consider the function 
$$
\psi(s):=p(1+s)^2\log (1+s)^2+q(1-s)^2\log (1-s)^2-[p(1+s)^2+q(1-s)^2]\log [p(1+s)^2+q(1-s)^2]-2s^2.
$$
Note that $\psi(0)=0$. It suffices to show $\psi(s)\le \psi(0), s\in [-1,1]$. For this we compute 
$$
\psi'(s)=4p(1+s)\log(1+s)-4q(1-s)\log(1-s)-2[p(1+s)+q(1-s)]\log [p(1+s)^2+q(1-s)^2]-4s.
$$
Since $\psi'(0)=0$, it is enough to prove $\psi''(s)\le 0, s\in [-1,1]$. We continue to compute 
$$
\psi''(s)=4p\log(1+s)+4q\log(1-s)-2\log [p(1+s)^2+q(1-s)^2]-\frac{4(s+p-q)^2}{p(1+s)^2+q(1-s)^2}.
$$
Therefore, the proof will be finished if for all $s\in [-1,1]$
$$
\phi(p):=4p\log(1+s)+4q\log(1-s)-2\log [p(1+s)^2+q(1-s)^2]\le 0,\qquad p\in [0,1].
$$
To see this, note that $\phi(0)=\phi(1)=0$. So it remains to prove the convexity of $\phi$, which follows from
$$
\phi'(p)=4\log(1+s)-4\log(1-s)-\frac{8s}{p(1+s)^2+q(1-s)^2}
$$
and 
$$
\phi''(p)=\frac{32s^2}{[p(1+s)^2+q(1-s)^2]^2}\ge 0.
$$

\section{More on the local Bobkov inequality}
\label{appendix on Bobkov}
Recall that in our proof for biased discrete hypercube case, we used \cite[Theorem 3.2]{bg99} stating that for probability spaces $(\Omega_i,\nu_i), 1\le i\le n$ and their product $(\Omega,\nu)$, we have 
\begin{equation}\label{ineq:bg general appendix}
\I(\E_{\nu}h)\le \sqrt{2}\E_{\nu}\sqrt{\sum_{i}\Var_{\nu_i}(h)}
\end{equation}
for any measurable  $h:\Omega\to \{0,1\}$, where $\I$ is the Gaussian isoperimetric profile $\I:[0,1]\to [0,1/\sqrt{2\pi}]$. In the following, we explain that the local Bobkov inequality extends to more general semigroups. 

Fix the probability spaces $(\Omega_i,\mu_i), 1\le i\le n$ and their product $(\Omega,\mu)$. Let $P_t^{(i)}$ denote the depolarizing Markov semigroup on $(\Omega_i,\mu_i)$, that is, $P_t^{(i)}$ has generator $L_{i}(f)=\mu_i(f)-f$. Denote by $P_t$ the product of $P_t^{(i)}$'s. We will show that for all $t\ge 0$, we have ($\Gamma
$ denoting the carr\'e du champs for $P_t$)
\begin{equation}\label{ineq:local bobkov general discrete}
    \I (P_t h)\le 2\sqrt{(1-e^{-t})}P_t \sqrt{\Gamma(h)},\qquad h:\Omega\to \{0,1\}.
\end{equation}
As before, we fix $t\ge 0$ and $x=(x_1,\dots, x_n)\in \Omega$. Consider the positive linear functional 
$$
\varphi_i(g):=P_t^{(i)}(g)(x_i),\qquad g:\Omega_i\to \real,
$$
such that $\varphi_i(1)=1$. Again, by Riesz representation theorem, there exists a probability measure $\nu_i$ on $\Omega_i$ such that 
$$
\E_{\nu_i}(g)=P_t^{(i)}(g)(x_i)=e^{-t}g(x_i)+(1-e^{-t})\mu_i(g),\qquad g:\Omega_i\to \real.
$$
Thus, for any $n\ge 1$ and $\nu:=\otimes_{i=1}^{n} \nu_i$ one has 
\begin{equation*}
    \E_{\nu}(g)=P_t(g)(x),\qquad g:\Omega\to \real.
\end{equation*}
Applying \eqref{ineq:bg general appendix} to the above $\nu_i$'s, one obtains 
\begin{equation*}
   \I(P_t h)(x)\le \sqrt{2}\E_{\nu}\sqrt{\sum_{i}\Var_{\nu_i}(h)(x)}.
\end{equation*}
It remains to show 
\begin{equation}\label{ineq:var gamma}
    \sum_{i}\Var_{\nu_i}(h)(x)\le 2(1-e^{-t})\Gamma (h)(x).
\end{equation}
To see this, note that if we denote by $\Gamma_i$ the carr\'e du champs associated to $P_t^{(i)}$, then 
\begin{equation*}
    \Gamma_i(h)=\frac{1}{2}\left[L_{i}(h^2)-2L_{i}(h)h\right]
    =\frac{1}{2}\left[\mu_i(h)-2\mu_i(h)h+h\right],
\qquad h:\Omega_i\to \{0,1\}
\end{equation*}
and 
\begin{equation*}
    \Gamma(h)
    =\frac{1}{2}\sum_{1\le i\le n}\left[\mu_i(h)-2\mu_i(h)h+h\right],
\qquad h:\Omega\to \{0,1\}.
\end{equation*}
A direct computation shows that for all $h:\Omega\to \{0,1\}$
\begin{align*}
   & \Var_{\nu_i}(h)=\E_{\nu_i}(h)-\E_{\nu_i}(h)^2\\
 &\qquad = (1-e^{-t})\left[e^{-t}(\mu_i(h)-h)^2+\mu_i(h)(1-\mu_i(h))\right]\\
 &\qquad\le (1-e^{-t})\left[(\mu_i(h)-h)^2+\mu_i(h)(1-\mu_i(h))\right]\\
 &\qquad=(1-e^{-t})\left[\mu_i(h)-2\mu_i(h)h+h\right].
\end{align*}
All combined, we proved \eqref{ineq:var gamma} and thus \eqref{ineq:local bobkov general discrete}.

\end{document}